\documentclass[11pt,a4paper]{amsart}[2000]
\title{Decomposable approximations of nuclear $C^*$-algebras}
\author{Ilan Hirshberg}
\address{Department of Mathematics, Ben Gurion University of the Negev, P.O.B. 653, Be'er Sheva 84105, Israel}
\email{ilan@math.bgu.ac.il}
\thanks{This research was supported by Israel Science Foundation grant 1471/07 and UK Engineering and Physical Sciences Research Council grant EP/I019227/1}

\author{Eberhard Kirchberg}
\address{Institut f{\"u}r Mathematik, Humboldt-Universit{\"a}t zu Berlin, 
Unter den Linden 6, D-10099 Berlin, Germany}
\email{kirchbrg@mathematik.hu-berlin.de}

\author{Stuart White}
\address{School of Mathematics and Statistics, University of Glasgow, Glasgow, G12 8QW, Scotland}
\email{stuart.white@glasgow.ac.uk}

\usepackage{amssymb}
\usepackage{amsthm}
\usepackage[all]{xypic}

\theoremstyle{plain}

\newtheorem{Thm}{Theorem}[section]

\newtheorem{Lemma}[Thm]{Lemma}

\theoremstyle{definition}
\newtheorem{Def}[Thm]{Definition}

\newtheorem{Rmk}[Thm]{Remark}
\newtheorem{Question}[Thm]{Question}

\newcommand{\B}{\mathcal{B}}
\newcommand{\A}{\mathcal{A}}
\newcommand{\M}{\mathcal{M}}

\newcommand{\K}{\mathcal{K}}

\newcommand{\Nh}{\mathcal{N}}

\newcommand{\Oh}{\mathcal{O}}

\newcommand{\N}{{\mathbb N}}

\newcommand{\lb}{\left <}
\newcommand{\rb}{\right >}

\newcommand{\eps}{\varepsilon}
\numberwithin{equation}{section}

\begin{document}
\begin{abstract}
We show that nuclear $C^*$-algebras have a refined version of the completely positive approximation property, in which the maps that approximately factorize through finite dimensional algebras are convex combinations of order zero maps. We use this to show that a separable nuclear $C^*$-algebra $\A$ which is closely contained in a $C^*$-algebra $\B$ embeds into $\B$.
\end{abstract}
\maketitle

The decomposition rank and nuclear dimension of a $C^*$-algebra are noncommutative notions  of covering dimension which play a prominent role in the structure and classification theory of $C^*$-algebras (see, e.g. \cite{et,winter-dr-Z-stability,winter-nd}). These dimensions are defined in terms of uniformly decomposable completely positive approximations of nuclear $C^*$-algebras.  The main theorem of this paper places these definitions in a broader context by providing a sharpening of the completely positive approximation property: nuclear $C^*$-algebras always have decomposable completely positive approximations.

We apply our approximation theorem to resolve a problem in perturbation theory of nuclear $C^*$-algebras.  Given two $C^*$-algebras $\A$ and $\B$ concretely represented on the same Hilbert space, say that $\A\subseteq_\gamma\B$ if operators in the unit ball of $\A$ can be approximated in $\B$ up to $\gamma$. In \cite{christensen}, Christensen showed that a sufficiently close near containment $\M\subseteq_\gamma\mathcal N$ of von Neumann algebras with $\M$ injective implies that there is an embedding $\M\hookrightarrow \mathcal N$. In the $C^*$-context, \cite[Theorem 6.10]{cssww1} (see also \cite{cssww2}) produces embeddings from sufficiently close near containments $\A\subseteq_\gamma\B$ when $\A$ is separable and has finite nuclear dimension, however the estimates in this result depend on the nuclear dimension of $\A$. The approximation theorem enables us to remove the hypothesis of finite nuclear dimension and establish a $C^*$-version of Christensen's embedding theorem with a universal constant valid for all separable nuclear $C^*$-algebras.  This is the subject of Section \ref{near inclusions}.

Recall that a $C^*$-algebra $\A$ is nuclear if and only if it has the completely positive approximation property (\cite{Choi-Effros-CPAP,kirchberg-CPAP}), that is, if for any finite subset $F \subseteq \A$ and for any $\eps>0$ there is a finite dimensional $C^*$-algebra $A_0$ and completely positive contractions
$$\xymatrix{
\A \ar[r]^{\psi} & A_0 \ar[r]^{\varphi} & \A
}$$ 
such that $\|\varphi \circ \psi (a) - a\|<\eps$ for all $a \in F$. We call the triple $(A_0,\psi,\varphi)$ a \emph{CP approximation} for $(F,\eps)$. The decomposition rank was introduced in \cite{kirchberg-winter} using order zero maps and decomposable maps. A completely positive contraction $\varphi: \A \to \B$ is said to be an \emph{order zero map} if $\varphi$ preserves orthogonality, i.e. whenever self-adjoint elements $x,y \in \A$ satisfy $xy = 0$ then we also have $\varphi(x)\varphi(y) = 0$. A completely positive map $\varphi:\A \to \B$ is said to be \emph{$n$-decomposable} if $\A$ decomposes into a direct sum $\A = \A_0 \oplus \A_1 \oplus \cdots \oplus \A_n$ such that $\varphi|_{\A_k}$ is an order zero map for all $k$. If $\varphi$ is $n$-decomposable for some $n$, we will say that $\varphi$ is \emph{decomposable}. A nuclear $C^*$-algebra $\A$ is said to have \emph{decomposition rank} $n$ if $n$ is the smallest number such that for every finite subset $F\subseteq\A$ and $\eps>0$, there exist CP approximation $(A_0,\psi,\varphi)$ where $\varphi$ is $n$-decomposable. The definition of nuclear dimension in \cite{winter-zacharias} uses a slightly weaker condition on approximating triples, asking that we can approximate $(F,\eps)$ by a triple $(A_0,\psi,\varphi)$ with $A_0$ finite dimensional, $\psi$ a completely positive contraction and $\varphi$ an $n$-decomposable map (but not necessarily contractive). This notion reaches purely infinite $C^*$-algebras, whereas $C^*$-algebras with finite decomposition rank are quasidiagonal and therefore stably finite (\cite{kirchberg-winter}).  

In Theorem \ref{ConvexCom}, we refine the completely positive approximation theorem and show that if $\A$ is nuclear, then one can always chose CP approximations such that the maps $\varphi$ going back into $\A$ are decomposable and contractive, only without an upper bound on the number of summands involved in the decomposition. In fact, the theorem is stronger - the map $\varphi$ can be chosen to be a convex combination of order zero maps. Thus, one can always choose decomposable CP approximations, and the definitions of decomposition rank and nuclear dimension require placing a uniform bound on the number of summands present.

In Section \ref{sec2}, we present an approximation theorem for weakly nuclear completely positive contractions $\Phi:\B\rightarrow\M$ from separable exact $C^*$-algebras into properly infinite von Neumann algebras. Such maps can be approximated in the point-weak$^*$ topology by a convex combination of $^*$-homomorphisms $\B\rightarrow\M$. In fact we can find a fixed $^*$-homomorphism $\varphi:\B\rightarrow\M$ so that $\Phi$ can be approximated by convex combinations of two unitary conjugates of $\varphi$.  As a consequence we obtain another proof of a special case of Theorem \ref{ConvexCom}, under the additional assumption that $\A^{**}$ is properly infinite (e.g. if $\A$ is stable or simple and purely infinite).

Part of the work on this paper was done while the authors visited the CRM in Barcelona and we thank the CRM for its hospitality. We thank Erik Christensen and the referee for their remarks on an earlier version of the paper. I.H. also wishes to thank Nate Brown 
for some helpful conversations regarding this paper.

\section{Approximation via convex combinations of order zero maps}

In this section we establish our strengthening of the completely positive approximation property for nuclear $C^*$-algebras. The first step is to establish a Kaplansky density lemma for order zero maps. Recall from \cite[Remark 2.4]{kirchberg-winter} that order zero maps are \emph{projective} in the sense that whenever $\mathcal J\lhd\mathcal D$ is an ideal and $\varphi:F\rightarrow\mathcal D/\mathcal J$ is an order zero map from a finite dimensional $C^*$-algebra, there exists an order zero map $\tilde{\varphi}:F\rightarrow \mathcal D$ lifting $\varphi$, i.e. $q\circ\tilde{\varphi}=\varphi$, where $q:\mathcal D\rightarrow \mathcal D / \mathcal J$ is the quotient map.   As noted in \cite{kirchberg-winter}, the projectivity of order zero maps arises from the correspondence between order zero maps $F\rightarrow\A$ and $^*$-homomorphisms from the cone on $F$ into $\A$ and Loring's projectivity of the cones on finite dimensional $C^*$-algebras (see \cite{loring}).

\begin{Lemma}\label{lemma:weak-stability}
Let $\A \subset \B(H)$ be a $C^*$-algebra, and let $\M$ be the strong$^*$-closure of $\A$. Let $F$ be a finite dimensional $C^*$-algebra and $\varphi:F\rightarrow \M$ be an order zero map. Then there exists a net $(\varphi_\lambda)$ of order zero maps $F\rightarrow \A$ such that $\varphi_\lambda(x)\rightarrow \varphi(x)$ in strong$^*$-topology for all $x\in F$.  In particular $\varphi_\lambda(x)\rightarrow\varphi(x)$ in weak$^*$-topology for all $x\in F$.
\end{Lemma}
\begin{proof}
Let $\Lambda$ denote the directed set of all strong$^*$-neighborhoods of $0\in H$ ordered by containment, i.e. for $V,W\in\Lambda$, write $V\geq W$ when $V\subseteq W$.  Let $\mathcal D$ consist of all bounded nets $(x_V)_{V\in\Lambda}$ in $\A$ indexed by $\Lambda$ which converge in the strong$^*$-topology to some $x\in\M$ (that is, such that for any strong$^*$-neighborhood $W$ of $0$ there exists $W_0 \geq W$ such that for any $V \geq W_0$ we have that $x_V - x \in W$).  
  Since multiplication is jointly strong$^*$-continuous on bounded sets, $\mathcal D$ forms a $^*$-subalgebra of $\ell^\infty(\Lambda,\A)$.  Given a sequence $(x^{(n)})_{n=1}^\infty$ in $\mathcal D$ converging (in norm) to $x\in\ell^\infty(\Lambda,\A)$, write $x^{(n)}_\infty$ for the strong$^*$-limit of $(x^{(n)}_V)_{V\in\Lambda}$. It is easy to check that $(x^{(n)}_\infty)_{n=1}^\infty$ is \emph{norm}-Cauchy, and therefore converges to $x_\infty$ in $\M$, say. By writing $x_V-x_\infty=(x_{V}-x^{(n)}_V)+(x^{(n)}_V-x^{(n)}_\infty)+(x^{(n)}_\infty-x_\infty)$, we see that $x_V$ converges in strong$^*$-topology to $x_\infty$ and so $\mathcal D$ is a $C^*$-subalgebra of $\ell^\infty(\Lambda,\A)$.

Now let $\mathcal J$ be the ideal in $\mathcal D$ consisting of strong$^*$-null nets and note that the argument above shows that $\mathcal J$ is norm closed in $\mathcal D$.  The quotient map $\mathcal D\rightarrow \mathcal D/\mathcal J$ is given by evaluating the strong$^*$-limit of a net in $\mathcal D$. By Kaplansky's density theorem every element of $\M$ is the strong$^*$-limit of a net indexed by $\Lambda$ and so $\mathcal D/\mathcal J$ is canonically $^*$-isomorphic to $\M$.

 Given a finite-dimensional $C^*$-algebra $F$ and an order zero map $\varphi:F\rightarrow\M$, projectivity enables us to lift $\varphi$ to an order zero map $\tilde{\varphi}:F\rightarrow\mathcal D$. Writing $\tilde{\varphi}$ as a net $(\varphi_\mathcal U)_{\mathcal U\in\Lambda}$ of contractions $F\rightarrow \A$, each of the maps $\varphi_\mathcal U$ is order zero and the lifting ensures that $\varphi_\mathcal U(x)\rightarrow\varphi(x)$ in strong$^*$-topology for all $x\in F$. 
\end{proof}

Our next lemma uses the fact that AF algebras are locally reflexive to show that hyperfinite von-Neumann algebras $\M$ (i.e. those containing a weak$^*$-dense AF $C^*$-subalgebra) have a stronger version of semi-discreteness: the maps coming back into $\M$ can be taken to be $^*$-homomorphisms.  If one defined a `semi-discreteness dimension' analogously to the nuclear dimension, then, as expected, all injective von Neumann algebras would be zero dimensional.
 Recall that a $C^*$-algebra $\B$ is said to be locally reflexive if for any finite dimensional operator system $E \subset \B^{**}$ there is a net of CP contractions $\psi_{\nu} : E \to \B$ such that $\psi_{\nu}(x) \to x$ weak$^*$ for any $x \in E$ (see \cite[Chapter 9]{Brown-Ozawa}). We recall that any nuclear $C^*$-algebra is locally reflexive, and in particular any AF algebra is locally reflexive. 
\begin{Lemma}\label{lemma:afd-weak-nuc}
Let $\M$ be a hyperfinite von-Neumann algebra. Then there exists a net of finite dimensional subalgebras $A_\nu$ and CP contractions $\Phi_\nu:\M \to A_\nu$ such that $\Phi_\nu(a) \to a$ in the weak$^*$ topology for any $a \in \M$.
\end{Lemma}
\begin{proof}
Given $x_1,\cdots,x_n\in \M$ and a weak$^*$-neighborhood $V$ of $0$, let $E$ be the finite dimensional operator system in $\M$ spanned by the $x_i$. It suffices to produce a finite dimensional subalgebra $A\subset \M$ and a completely positive contraction $\Phi:\M\rightarrow A$ with $\Phi(x_i)-x_i\in V$ for $i=1,\cdots,n$, as then we can produce a net of such maps indexed by finite subsets of $\M$ and a weak$^*$-neighborhood basis of $0$. Let $\B$ be a weak$^*$-dense AF subalgebra of $\M$ and regard $\M$ as a von Neumann subalgebra of $\B^{**}$. By local reflexivity there exists a completely positive contraction $\varphi:E\rightarrow B$ with $\varphi(x_i)-x_i\in V$ for all $i$. By choosing a finite dimensional subalgebra $A\subset \B$ which $\eps$-contains the image under $\varphi$ of the unit ball of $E$ for a sufficiently small $\eps$ and composing $\varphi$ by a conditional expectation from $\B$ onto $A$, we may additionally assume that $\varphi$ takes values in $A$.  Arveson's extension theorem then enables us to extend this map to the required map $\Phi$. 
\end{proof}

The previous lemmas together with Connes' theorem from \cite{Connes.InjectiveFactors} that injectivity implies hyperfiniteness will now be used to derive an approximation property for nuclear $C^*$-algebras in the weak topology in which the maps going into the algebra are order zero.
\begin{Lemma}\label{L2}
If $\A$ is a separable nuclear $C^*$-algebra, then there is a net of CP contractions 
$$\xymatrix{
\A \ar[r]^{\psi_\nu} 
& A_\nu \ar[r]^{\varphi_\nu} & \A
}$$ 
 with $A_{\nu}$ finite dimensional such that for all $a \in \A$, $\varphi_{\nu} \circ \psi_{\nu}(a) \to a$ weakly and $\varphi_{\nu}$ are order zero maps. 
\end{Lemma}
\begin{proof}
Since $\A$ is nuclear, $\A^{**}$ is injective. If $\pi$ is a GNS representation of $\A$, then $\pi(\A)''$ is a cutdown of $\A^{**}$ by a central projection, and therefore is injective as well. Since $\A$ is separable, $\pi(\A)''$ has separable predual, and hence is hyperfinite by Connes' theorem. Since $\A^{**}$ is a direct sum of (possibly uncountably many) hyperfinite $W^*$-algebras, it is hyperfinite itself (see \cite[Chapter XVI]{takesaki-III}). 

For $a_1,\cdots,a_n\in \A$ and a weak neighborhood $V$ of $0$ in $A$, take a weak neighborhood $W$ of $0$ with $W+W\subseteq V$. By Lemma \ref{lemma:afd-weak-nuc}, find a finite dimensional subalgebra $A_0\subset \A^{**}$ and a CP contraction $\psi:\A^{**}\rightarrow A_0$ such that $\psi(a_i)-a_i\in W$ for $i=1,\cdots,n$.   Applying Lemma \ref{lemma:weak-stability} to the inclusion map $A_0\hookrightarrow \A^{**}$, we can find an order zero map $\varphi:A_0\rightarrow\A$ with $\varphi(\psi(a_i))-\psi(a_i)\in W$ for $i=1,\cdots,n$. Thus $\varphi(\psi(a_i))-a_i\in V$ for $i=1,\cdots,n$. Using nets indexed by finite subsets of $\A$ and weak-neighborhoods of $0$ gives the result.
\end{proof}

We are now in a position to deduce the main result of the paper from the preceding lemmas. This works in the same way as the completely positive approximation property of a $C^*$-algebra is deduced from semidiscreteness of its bidual.  
\begin{Thm}\label{ConvexCom}
If $\A$ is a nuclear $C^*$-algebra, then for any finite set $F \subseteq \A$ and any $\eps>0$ there is a CP-approximation $(A_0,\psi,\varphi)$, $\psi:\A \to A_0$, $\varphi:A_0 \to \A$ for $F$ to within $\eps$ such that $\varphi,\psi$ are CP contractions, $A_0$ is finite dimensional, and $\varphi$ is a convex combination of finitely many contractive order zero maps. 
\end{Thm}
\begin{proof}
Since any finite subset of a nuclear $C^*$-algebra is contained in a separable nuclear subalgebra (see \cite[Exercise 2.3.9]{Brown-Ozawa}), we may assume without loss of generality that $\A$ is separable. 

It follows from the Hahn-Banach theorem that any convex subset of $\B(\A)$ has the same point-norm and point-weak closures. Thus, let $K_0 \subseteq \B(\A)$ be the set of all contractive CP maps $T:\A \to \A$ which admit a factorization $T = \varphi \circ \psi$, where $\A \overset{\psi}{\longrightarrow} A_0 \overset{\varphi}{\longrightarrow} \A$, with $A_0$ finite dimensional, $\psi$ a CP contraction and $\varphi$ an order zero map, and let $K = \textrm{conv}(K_0)$. By Lemma \ref{L2}, the identity map $\A\rightarrow \A$ lies in the point-weak closure of $K_0$ and so there is a net of elements $T_\lambda \in K$ such that $T_\lambda(a) \to a$ in norm for all $a \in \A$. 

We finally note that any $T \in K$ can be decomposed as
$$\xymatrix{
\A \ar[r]_{\psi} \ar@/^1pc/[rr]^{T} & B \ar[r]_{\varphi} & \A
}$$ 
with $B$ finite dimensional and $\varphi$ a finite convex combination of order zero maps, as follows. Write $T = \sum_{i=1}^n t_iT_i$, a convex combination of elements $T_i \in K_0$. Decompose each $T_i$ as  
$$\xymatrix{
\A \ar[r]_{\psi_i} \ar@/^1pc/[rr]^{T_i} & B_i \ar[r]_{\varphi_i} & \A
}$$
with $\varphi_i$ order zero. Setting $B = \bigoplus_{i=1}^n B_i$, $\psi(a) = \bigoplus_{i=1}^n\psi_i(a)$ and, denoting $\tilde{\varphi}_i (a_1 \oplus a_2 \oplus ... \oplus a_n) = \varphi_i(a_i)$, we can take $\varphi =  \sum_{i=1}^n t_i\tilde{\varphi}_i$ to obtain the required decomposition of $T$. 
\end{proof}

\section{Near Inclusions} \label{near inclusions}

In this section we apply Theorem \ref{ConvexCom} to near containments. First we recall the definition of a near inclusion.

\begin{Def}
Let $\A$ and $\B$ be $C^*$-subalgebras of $\B(H)$. For $\gamma>0$, write $\A\subseteq_\gamma\B$ if for each $a\in\A$, there exists $b\in\B$ with $\|a-b\|\leq\gamma \|a\|$. Write $\A\subset_\gamma\B$ if there exists $\gamma'<\gamma$ with $\A\subseteq_\gamma\B$.
\end{Def}

Combining Theorem \ref{ConvexCom} with a perturbation theorem for order zero maps from \cite{cssww1}, we can linearize a near inclusion $\A\subset_\gamma\B$ when $\A$ is nuclear.

\begin{Lemma}\label{lin}
Let $\A\subset_\gamma\B$ be a near inclusion of $C^*$-algebras and suppose that $\A$ is nuclear. Let $\eta=2(2\gamma+\gamma^2)(2+2\gamma+\gamma^2)$. Then there exists $\eta_0<\eta$ with the property that for each finite subset $Z$ of the unit ball of $\A$, there exists a contractive CP map $\Phi:\A\rightarrow\B$ such that $\|\Phi(z)-z\|\leq \eta_0
$ for $z\in Z$.\end{Lemma}
\begin{proof}
Fix $\gamma_0<\gamma$ so that $\A\subseteq_{\gamma_0}\B$ and choose $\eps>0$ such that
$$
\eta_0=2(2\gamma_0+\gamma_0^2)(2+2\gamma_0+\gamma_0^2)+\eps<\eta.
$$
Given a finite subset $Z$ of the unit ball of $\A$, use Theorem \ref{ConvexCom} to produce a CP approximation $(A_0,\psi,\varphi)$, $\psi:\A\rightarrow A_0$, $\varphi:A_0\rightarrow\A$ for $Z$ to within $\eps$ such that $\varphi$ and $\psi$ are CP contractions, $A_0$ is finite dimensional, and $\varphi=\sum_{i=1}^n\lambda_i\varphi_i$ is a convex combination of order zero maps. For each $i$, use \cite[Theorem 6.4]{cssww1} to find a CP map $\tilde{\theta_i}:A_0\rightarrow\B$ with $\|\tilde{\theta_i}-\varphi_i\|_{cb}\leq (2\gamma_0+\gamma_0^2)(2+2\gamma_0+\gamma_0^2)$. By rescaling each $\tilde{\theta_i}$ if necessary we can find contractive CP maps $\theta_i:A_0\rightarrow\B$ with $\|\varphi_i-\theta_i\|_{cb}\leq 2\|\varphi_i-\tilde{\theta_i}\|_{cb}$. Define $\theta=\sum_{i=1}^n\lambda_i\theta_i$. This is a contractive CP map with $\|\theta-\varphi\|_{cb}\leq  2(2\gamma_0+\gamma_0^2)(2+2\gamma_0+\gamma_0^2)$. Thus $\Phi=\theta\circ\psi$ is a contractive CP map from $\A$ into $\B$ with
\begin{equation*}
\|\Phi(z)-z\|\leq\|\theta(
\psi(z))-\varphi(\psi(z))\|+\|\varphi(\psi(z))-z\|\leq \eta_0,\quad z\in Z.\qedhere
\end{equation*}
\end{proof}

The embedding theorem below now follows immediately from \cite[Lemma 4.1]{cssww1} (taking the $\mu$ of that lemma to be small enough that $8\sqrt{6}(\eta_0)^{1/2}+\eta_0+\mu<8\sqrt{6}\eta^{1/2}+\eta$).

\begin{Thm}\label{ni}
Let $\A\subset_\gamma\B$ be a near inclusion of $C^*$-algebras for some $\gamma$ satisfying
$$
\eta=2(2\gamma+\gamma^2)(2+2\gamma+\gamma^2)< 1/210000.
$$
Suppose that $\A$ is separable and nuclear. Then for each finite subset $X$ of the unit ball of $\A$, there exists an embedding $\theta:\A\hookrightarrow\B$ with $\|\theta(x)-x\|\leq 8\sqrt{6}\eta^{1/2}+\eta$, for $x\in X$.
\end{Thm}
 
Note that the examples of Johnson \cite{johnson} show that in the situation of Theorem \ref{ni} it is not generally possible to construct embeddings $\theta:\A\rightarrow\B$ which are uniformly close to the inclusion of $\A$ into the underlying $\B(H)$.

\section{Completely positive maps from exact $C^*$-algebras}\label{sec2}

In this last section, we examine weakly nuclear completely positive maps from exact algebras into properly infinite von Neumann algebras.  Via the $\Oh_2$-embedding theorem, such maps are automatically nuclear (see Remark \ref{rem} below). More surprisingly we can approximate all such maps by convex combinations of $^*$-homomorphisms. In fact more is true. An average of two unitary conjugates of a given $^*$-homomorphism can be used to make this approximation. This can be used to give a different proof of Theorem \ref{ConvexCom} under the stronger assumption that $\A^{**}$ is properly infinite (as happens when $\A$ is stable). This alternative approach avoids using Connes' theorem on the equivalence of injectivity and hyperfiniteness. We begin by recalling the standard fact that properly infinite von Neumann algebras absorb $\B(\ell^2)$ tensorially.
\begin{Lemma}
\label{Lemma:absorb-B(H)}
Let $\M$ be a properly infinite von-Neumann algebra.  Then $\M \cong \M\,\overline{\otimes}\,\B(\ell^2)$.
\end{Lemma}
\begin{proof}
Since $1$ is properly infinite, $\M$ has a countably infinite family of mutually equivalent orthogonal projections which sum up to 1. Denoting one of those projections by $e$, it follows from \cite[Proposition V.1.22]{takesaki-I},  that $\M \cong e\M e\,\overline{\otimes}\,\B(\ell^2)$. Since $\B(\ell^2) \cong  \B(\ell^2)\,\overline{\otimes}\,\B(\ell^2)$, we see that 
\begin{equation*}
\M \cong e\M e\,\overline{\otimes}\,\B(\ell^2)\,\overline{\otimes}\,\B(\ell^2) \cong \M\,\overline{\otimes}\, \B(\ell^2).\qedhere
\end{equation*}
\end{proof}

We next prove a dilation lemma for unital completely positive maps into properly infinite von-Neumann algebras. In this lemma and the following theorem, we write $\sim$ for the Murray-von Neumann equivalence relation on projections.  

\begin{Lemma}
\label{Lemma:dilation}
Let $\A$ be a separable unital $C^*$-algebra, $\M$ be a properly infinite von Neumann algebra and $\varphi:\A\rightarrow \M$ a unital completely positive map.  Given a projection $p\in \M$ with $1\sim p\sim 1-p$, there exists a $^*$-homomorphism $\theta:\A\rightarrow \M$ with $p\theta(x)p=p\varphi(x)p$ for all $x\in \A$.
\end{Lemma}
\begin{proof}
First note that the multiplier algebra $M(\K\otimes \M)$ of $\K\otimes \M$ is contained in $\B(\ell^2)\,\overline{\otimes}\,\M$.  Indeed, representing $\M$ non-degenerately on a Hilbert space $H$ so that $\K\otimes \M$ is non-degenerately represented on $\ell^2\otimes H$, take an increasing sequence $(P_n)_n$ of finite rank projections in $\K$ converging strongly to $1_{\ell^2(\N)}$. Given any idealizer $x\in\B(\ell^2 \otimes H)$ (i.e. $x$ satisfying $x(\K\otimes \M),(\K\otimes \M)x\subseteq \K\otimes \M$), it follows that $(P_n\otimes 1)x(P_n\otimes 1)\in \K\otimes \M$. By taking strong operator limits we see that $x\in \B(\ell^2)\,\overline{\otimes}\,\M$.

Let $e$ be a rank one projection in $\K$ and, by identifying $\M$ with the subalgebra $(e\otimes 1)(\K\otimes \M)(e\otimes 1)$ of $\K\otimes \M$,  view $\varphi$ as a completely positive map into $M(\K\otimes \M)$.  Write $p_0=e\otimes p$, and apply Kasparov's Stinespring theorem \cite[Theorem 3]{Kas} to the map $x\mapsto p_0\varphi(x)p_0$ to obtain a partial isometry $v\in M(\K\otimes \M)$ and a $^*$-homomorphism $\pi:\A\rightarrow M(\K\otimes \M)$ with $p_0\varphi(x)p_0=v\pi(x)v^*$ for $x\in \A$. By the previous paragraph, we can view $\pi$ as taking values in $\B(\ell^2)\,\overline{\otimes}\, \M$ and $v$ as a partial isometry in this algebra. We may assume that $v=v\pi(1)$ so that $vv^*=p_0$. Set $q_0=v^*v$.

Write $\Nh_0=\B(\ell^2)\,\overline{\otimes}\,\M$. Under the identification of $\M$ with $\B(\ell^2)\,\overline{\otimes}\,\M$ of Lemma \ref{Lemma:absorb-B(H)}, we see that $e\otimes 1_\M$ is identified with $e\otimes 1_{\B(\ell^2)}\otimes 1_\M\sim 1_{\B(\ell^2)}\otimes 1_{\B(\ell^2)}\otimes 1_\M$. Undoing this identification, we have $e\otimes 1_\M\sim 1_{\B(\ell^2)}\otimes 1_\M$ in $\Nh_0$.  Thus $q_0\sim p_0=e\otimes p\sim e\otimes 1_\M\sim 1_{\B(\ell^2)}\otimes 1_\M$.  Now consider $\Nh_1=\B(\ell^2)\,\overline{\otimes}\,\Nh_0$ and let $p_1=e\otimes p_0$, $q_1=e\otimes q_0$. Arguing as above, we see that $p_1\sim q_1\sim 1_{\Nh_1}$. Furthermore
\begin{align*}
1_{\Nh_1}-q_1=&\,(1_{\B(\ell^2)}-e)\otimes q_0+(1_{\B(\ell^2)}\otimes (1_{\Nh_0}-q_0))\\
 \sim&\,1_{\B(\ell^2)}\otimes q_0+(1_{\B(\ell^2)}\otimes (1_{\Nh_0}-q_0))=1_{\Nh_1}.
\end{align*}
In conclusion, we obtain the equivalence $e\otimes e\otimes (1_\M-p)\sim e\otimes e\otimes 1_\M\sim1_{\Nh_1}\sim 1_{\Nh_1}-q_1$ and so there is a partial isometry $w\in \Nh_1$ with $ww^*=e\otimes e\otimes(1_{\M}-p)$ and $w^*w=1_{\Nh_1}-q_1$.  Since $e\otimes v$ and $w$ have orthogonal domain projections and orthogonal range projections, $u=e\otimes v+w$ is an isometry in $\Nh_1$ with $uu^*=e\otimes e\otimes 1_{\M}$.  Define a $*$-homomorphism $\theta:\A\rightarrow (e\otimes e\otimes 1_{\M})\Nh_1(e\otimes e\otimes 1_{\M})$ by $\theta(x)=u(1\otimes \pi(x))u^*$.  As $(e\otimes e\otimes p)u=e\otimes v$, we have
$$
(e\otimes e\otimes p)\theta(x)(e\otimes e\otimes p)=(e\otimes v)(1\otimes \pi(x))(e\otimes v)^*=e\otimes p_0\varphi(x)p_0,\quad x\in \A.
$$
Identifying $\M$ with $(e\otimes e\otimes 1_{\M})\Nh_1(e\otimes e\otimes 1_{\M})$, we can view $\theta$ as a $^*$-homomorphism from $\A$ into $\M$ with $p\theta(x)p=p\varphi(x)p$ for $x\in \A$, as required.
\end{proof}

We can now establish our approximation result for weakly nuclear completely positive contractions from separable exact algebras into properly infinite von Neumann algebras.  Given a state $\lambda$ on a von Neumann algebra $\M$, we write $\|x\|_\lambda$ for $\lambda(x^*x)^{1/2}$.

\begin{Thm}\label{exact-embedding}
Let $\M$ be a properly infinite von-Neumann algebra and let $\B$ be a separable exact $C^*$-algebra. Let $\Phi:\B \to \M$ be a weakly nuclear completely positive contraction. 
\begin{enumerate}
\item There exists an injective $^*$-homomorphism $\varphi:\B \to \M$ such that $\Phi$ is in the point-weak$^*$ closure of the convex hull of the family of maps $\{u^* \varphi (\cdot) u \mid u \in U(\M)\}$.  The $^*$-homomorphism $\varphi$ can be taken to be unital when $\Phi$ is unital.
\item
\label{furthermore} Furthermore, if $\A \subseteq \M$ is a weak$^*$-dense $C^*$-algebra such that $\Phi(\B) \subseteq \A$ and $\Oh_2$ is unitally contained in the multiplier algebra $M(\A)\subseteq \M$ (e.g. if $\A$ is stable) then one can find an injective $^*$-homomorphism $\varphi:\B\rightarrow M(\A)$ such that $\Phi$ is in the point-$\sigma(M(\A),\M_*)$-closure of the convex hull of the family of maps $\{e^{-ih}\varphi(\cdot)e^{ih} \mid h \in \A_{s.a.}, \|h\|<\pi\}$.  Again $\varphi$ can be taken to be unital when $\Phi$ is unital.
\end{enumerate}
\end{Thm}
\begin{proof}
We first reduce to the case in which both $\B$ and $\Phi$ are unital. If $\B$ is not unital, then we can add a unit to $\B$ and extend $\Phi$ uniquely to a unital CP map from the unitization. If $\B$ was already unital but $\Phi$ is not, then we can adjoin a new unit to $\B$ and extend $\Phi$ to a unital map as well (see \cite[Section 2.2]{Brown-Ozawa}, for example). We thus assume that $\B$ is unital and $\Phi$ is unital.

Next, we reduce to the case where $\B$ is the Cuntz algebra $\Oh_2$. By \cite[Theorem 2.8]{kirchberg-phillips}, we can find a unital embedding of $\B$ in $\Oh_2$. Since $\Phi$ is weakly nuclear, we can take a net $(\Phi_\tau)$ of CP-approximations to $\Phi$ in the point-weak$^*$ topology which factorize through finite dimensional algebras as
$$\xymatrix{
\B \ar[r]_{S_{\tau}} \ar@/^1pc/[rr]^{\Phi_{\tau}} & M_{n_{\tau}} \ar[r]_{T_{\tau}} & \M
}$$
where $S_{\tau},T_{\tau}$ are CP and unital. Use Arveson's extension theorem to extend $S_{\tau}$ to a unital CP map $\tilde{S}_{\tau}:\Oh_2 \to M_{n_{\tau}}$. Then any point-weak$^*$ cluster point of $(T_\tau\circ\tilde{S}_\tau)$ is an extension of $\Phi$ to $\Oh_2$. Thus we may assume without loss of generality that $\B = \Oh_2$. 
  
Since $\M$ is properly infinite, $\Oh_2$ embeds unitally into $\M$.  Fix such a unital embedding $\varphi:\Oh_2\rightarrow\M$. We will show that convex combinations of unitary conjugates of $\varphi$ can be used to approximate $\Phi$. To this end, fix a finite dimensional operator system $X\subseteq \M$, normal states  $\rho_1,...\rho_n \in \M_*$ and $\eps>0$. We will find unitaries $u_1,u_2\in U(\M)$ such that
$$
\left|\rho_i\left(\frac{1}{2}(u_1^*\varphi(x)u_1+u_2^*\varphi(x)u_2)-\Phi(x)\right)\right|<\eps\|x\|,\quad x\in X.
$$
Since $X$, $\varepsilon$ and the states $\rho_1,\cdots,\rho_n$ are arbitrary, this will establish the result.  Note that it actually proves more; $\Phi$ can be approximated by a convex combination of two unitary conjugates of $\varphi$.

By the Radon-Nikodym theorem for states (see \cite[Proposition 3.8.3]{Brown-Ozawa}), there is a normal state $\lambda\in\M_*$  such that in the corresponding GNS representation $\pi_{\lambda}$ with cyclic vector $\xi_{\lambda}$ we can find positive operators $y_1,...,y_n \in \pi_{\lambda}(\M)'$ satisfying $\rho_i(x) = \lb \pi_{\lambda}(x)\xi_{\lambda},y_i\xi_{\lambda} \rb$ for $i=1,\cdots,n$.  Fix $\eps'>0$ such that $3\|y_i\|\sqrt{\eps'}+\eps'<\eps$ for all $i$.  As $\M$ is properly infinite, there exists a projection $p_1\in\M$ with $p_1\sim 1-p_1\sim 1$ such that
\begin{equation}\label{eq:2.3-1}
\|(1-p_1)\Phi(x)\|_\lambda^2=\lambda(\Phi(x)^*(1-p_1)\Phi(x))\leq \eps'\|x\|^2,\quad x\in X.
\end{equation}
By Lemma \ref{Lemma:dilation}, there is a $^*$-homomorphism $\kappa:\Oh_2 \to \M$ such that
\begin{equation}\label{eq:2.3-2}
p_1\kappa(a)p_1 = p_1\Phi(a)p_1,\quad a\in\Oh_2.
\end{equation}

Let $s = 1-2p_1$, then $s$ is a self-adjoint unitary and we have 
$$
\frac{1}{2} (s^*\kappa(a) s + \kappa(a)) = p_1\kappa(a)p_1 + (1-p_1)\kappa(a)(1-p_1),\quad a\in\Oh_2.
$$
For any $y\in\M$, the estimate (\ref{eq:2.3-1}) gives
$$
\|y(1-p_1)\|^2_{\lambda} = \lambda((1- p_1)y^*y(1-p_1) ) \leq \|y\|^2\lambda(1-p_1)  \leq\|y\|^2\eps',
$$
as $1\in X$. Using (\ref{eq:2.3-2}), we have 
\begin{align*}
&\,\|(1-p_1)\kappa(x)(1-p_1)+p_1\kappa(x)p_1 - \Phi(x)\|_{\lambda}\\
\leq&\,\|(1-p_1)\kappa(x)(1-p_1)\|_\lambda+\|(1-p_1)\Phi(x)\|_{\lambda} + \|p_1\Phi(x)(1-p_1)\|_{\lambda}\leq\,3\|x\|\sqrt{\eps'},
\end{align*} 
for any $x \in X$,
using (\ref{eq:2.3-1}) to estimate the middle term and the previous estimate for the first (with $y=(1-p_1)\kappa(x)$) and third (with $y=p_1\Phi(x)$). Therefore
$$
|\rho_i(p_1\kappa(x)p_1+(1-p_1)\kappa(x)(1-p_1) - \Phi(x))| \leq 3\|x\|\|y_i\|\sqrt{\eps'},\quad i=1,\cdots n,\ x\in X.
$$

Since any two unital embeddings $\Oh_2\rightarrow\M$ are approximately unitary equivalent (as any properly infinite von Neumann algebra $\M$ satisfies the hypotheses of \cite[Theorem 3.6]{rordam}), we can choose a unitary $u \in U(\M)$ such that 
$$
\|u^*\varphi(x)u - \kappa(x)\|<\eps'\|x\|,\quad x\in X.
$$
Taking $u_1=us$ and $u_2=u$, we have 
\begin{align*}
\left|\rho_i\left(
\frac{1}{2} (u_1^*\varphi(x)u_1 + u_2^*\varphi(a)u_2) - \Phi(x) 
\right)\right| \leq&\,\left|\rho_i\left(\frac{1}{2}(s^*\kappa(x)s+\kappa(x))-\Phi(x)\right)\right|+\eps'\|x\|\\
\leq&\,3\|x\|\|y_i\|\sqrt{\eps'} + \|x\|\eps'\\
\leq&\,\|x\|\eps,\quad i=1,\cdots,n,\ x\in X,
\end{align*}
as claimed.

For part \ref{furthermore}, note that we had the freedom to choose any embedding $\varphi:\Oh_2\rightarrow\M$ and so we can ensure that the range of $\varphi$ lies in $M(\A)$.  As for the unitaries, as $\M$ is a von-Neumann algebra, we can write the unitaries in the form $e^{ib}$ for $b \in \M_{sa}$ with $\|b\|\leq \pi$. By Kaplansky's density theorem, we can find a net of self-adjoint elements $h_{\tau}\in\A$ with norm smaller than $\pi$ which converge to $b$ in the SOT. Since the exponential function is SOT-continuous (see \cite[Lemma I.7.2]{davidson}, for example), it follows that $e^{ih_{\tau}} \to e^{ib}$ in SOT, and therefore one could replace the unitaries by exponentials as in the statement.  
\end{proof}

We now show how the previous result gives a second proof of Theorem \ref{ConvexCom} under the stronger assumption that $\A^{**}$ is properly infinite (e.g. if $\A$ is stable or simple and purely infinite). To see this, note that if $F$ is a finite dimensional $C^*$-algebra then any completely positive contraction $\Phi:F \to \A \subseteq \A^{**}$ is in the point-weak$^*$ closure of the convex hull of the $^*$-homomorphisms from $F$ to $\A^{**}$ by Theorem \ref{exact-embedding}. Using Lemma \ref{lemma:weak-stability}, it follows that $\Phi$ is in the  point-weak closure of the convex hull of the \emph{order zero maps} from $F$ to $\A$. As before, it follows from the Hahn-Banach theorem that in fact $\Phi$ is in the point-\emph{norm} closure of convex combinations of order zero maps from $F$ to $\A$. 

Since $\A$ is nuclear, we can find a net of CP approximations  $\Psi_{\lambda}:\A \to F_{\lambda}$, $\Phi_{\lambda}:F_{\lambda} \to \A$ to the identity. By the argument above, we may assume without loss of generality that $\Phi_{\lambda}$ is in fact a convex combination of order zero maps, giving us the required result.

\begin{Rmk}\label{rem}
The fact that separable exact algebras embed in $\Oh_2$ can be used to give another characterization of exactness as follows.  A $C^*$-algebra $\B$ is exact if and only every weakly nuclear CP contraction $\Phi$ from $\B$ into a von Neumann algebra $\M$ is in fact nuclear.  
\end{Rmk}
\begin{proof}
If $\B$ satisfies the hypothesis of the remark, then take $\Phi$ to be any faithful embedding of $\B$ into $\B(H)$ to see that $\B$ is nuclearly embeddable and therefore exact (see \cite{Brown-Ozawa}).  Conversely, let $\B$ be exact and take a weakly nuclear CP contraction $\Phi:\B\rightarrow \M$.  Fix a finite dimensional operator system $E\subset \B$ and consider the separable $C^*$-algebra $\B_0$ generated by $E$ which is also exact.  Thus we can embed $\B_0$ in $\Oh_2$.   As in the first part of the proof of Theorem \ref{exact-embedding}, $\Phi|_E$ can be extended to $\Oh_2$ and so is nuclear as $\Oh_2$ is nuclear.  Since $E$ was arbitrary, we can approximate $\Phi$ in the point-norm topology.
\end{proof}

Theorem \ref{exact-embedding}, requires $\M$ to be properly infinite.  It is natural to ask what happens when $\M$ is a finite von Neumann algebra. 
\begin{Question}
Let $\M$ be type II$_1$ factor, let $F$ be a finite dimensional $C^*$-algebra and let $\Phi:F \to \M$ be a completely positive contraction. Can $\Phi$ be approximated in the point-weak topology by convex combinations of order zero maps?
\end{Question}
One could ask the question about finite von-Neumann algebras which are not factors and do not have a type I part. One should note, though, that if $\M$ is finite dimensional then one would not necessarily be able to find such approximations for completely positive maps, due to dimension restrictions. For example, if $n>m$ there are many CP maps from $M_n$ to $M_m$ but no non-zero order zero maps. This question is also interesting when we restrict to trace preserving maps. Note that different factorization properties of trace preserving contractive CP maps between finite von Neumann algebras have been considered in the literature and can be somewhat delicate, see \cite{haag}, for example.

\newcommand{\etalchar}[1]{$^{#1}$}

\end{document}